\documentclass[11pt]{article}      

\usepackage{lineno}
\usepackage{mathrsfs}
\usepackage{graphicx}
\usepackage{mathtools}
\usepackage{caption,subcaption}
\usepackage{amsthm}
\usepackage{textcomp}
\usepackage{euscript}
\usepackage{eufrak}
\usepackage{xcolor}
\usepackage{tcolorbox}
\usepackage[titletoc,title]{appendix}
\usepackage{upgreek}
\usepackage{hyperref}
\usepackage{float}
\usepackage{longtable}
\newcommand*{\email}[1]{%
    \normalsize\href{mailto:#1}{#1}\par
    }
    
\usepackage{tikz}
\usepackage{amsmath,bm,bbm, amssymb}
\usepackage{fullpage}
\usepackage{color}
\usepackage{animate}
\usepackage{etoolbox}
\usepackage{ulem}
\usepackage{comment}
\usepackage{pgf}
\usepackage[ruled,vlined,linesnumbered,algosection,resetcount]{algorithm2e}
\usetikzlibrary{calc}
\usetikzlibrary{patterns}
\usetikzlibrary{arrows}
\usetikzlibrary{decorations.pathreplacing}
\usepackage[scaled]{helvet}
\usepackage[utf8]{inputenc}
\usepackage{pgfplots}
\definecolor{wiasblue}   {cmyk}{1.0, 0.60, 0, 0}
\definecolor{mlugreen}{RGB}{0,81,51}

\providecommand{\keywords}[1]
{
  \small	
  \textbf{\textit{Keywords---}} #1
}
\def\Z{\mathbb Z}
\def\E{\mathbb E}
\def\P{\mathbb P}
\def\Q{\mathbb Q}
\def\R{\mathbb R}
\def\mc{\mathcal}
\def\ms{\mathsf}

\def\la{\lambda}

\def\g{\gamma}
\def\a{\alpha}
\def\k{\kappa}

\def\es{\emptyset}
\def\one{\mathbbmss{1}}
\def\ms{\mathsf}
\def\mc{\mathcal}

\def\Pois{F_{\ms{Poi}}}
\def\Str{\ms{Str}}
\def\d{\,\mathrm d}

\def\Xs{X^{\ms S}}
\def\Xm{X^{\ms M}}
\def\Xl{X^{\ms L}}

\def\pc{p^{\ms{Cond}}}

\def\pcmc{p^{\ms{CMC}}}
\def\pp{q^{\ms{IS}}}

\def\b{\beta}
\def\lap{\la^{\ms{PS}}}
\def\rp{\rho^{\ms{PS}}}

\def\ff{\infty}
\def\s{\sigma}

\def\FF{\mc F}

\def\lw{\leftarrow}

\newtheorem{theorem}{Theorem}

\newtheorem{remark}{Remark}

\title{Rare Events in Random Geometric Graphs}
\author{
  Christian Hirsch\\ {\normalfont University of Groningen}\\ {\normalfont \email{c.p.hirsch@rug.nl}}\\
  \and
  Sarat B. Moka \\ {\normalfont University of Queensland}\\ {\normalfont \email{s.babumoka@uq.edu.au}}\\
  \and
  Thomas Taimre \\{\normalfont University of Queensland}\\ {\normalfont \email{t.taimre@uq.edu.au}}\\
  \and
  \ \\ Dirk P. Kroese\\ {\normalfont University of Queensland }\\ {\normalfont \email{kroese@maths.uq.edu.au}}\\
}
\begin{document}

\maketitle

\begin{abstract}
        This work introduces and compares approaches for estimating rare-event probabilities related to the number of edges in the random geometric graph on a Poisson point process. In the one-dimensional setting, we derive closed-form expressions for a variety of conditional probabilities related to the number of edges in the random geometric graph and develop conditional Monte Carlo algorithms for estimating rare-event probabilities on this basis. We prove rigorously a reduction in variance when compared to the crude Monte Carlo estimators and illustrate the magnitude of the improvements in a simulation study. In higher dimensions, we leverage conditional Monte Carlo to remove the fluctuations in the estimator coming from the randomness in the Poisson number of nodes. Finally, building on conceptual insights from large-deviations theory, we illustrate that importance sampling using a Gibbsian point process can further substantially reduce the estimation variance.
\end{abstract}
\vspace{-.1cm}

\maketitle
%
%
\section{Introduction}
\label{intrSec}

%
%

In this paper, we focus on rare events associated with the number of edges in the {\em Gilbert graph} $G(X)$ on a homogeneous Poisson point process
$X = \{X_i\}_{i \ge 1}$ with intensity $\la > 0$ in $\R^d$, for $d \ge 1$. Consequently, the nodes of $G(X)$ are the points of $X$ and there is an edge between $X_i, X_j \in X$ if $\|X_i - X_j\| \le 1$, where the upper bound $1$ is the \emph{threshold} of the Gilbert graph and $\|\cdot \|$ denotes the Euclidean norm.  Our goal is to analyze the probability of the rare event that the number of edges in a bounded sampling window $W \subset \R^d$ deviates considerably from its expected value. More succinctly, we ask:
\bigskip
\begin{center}
	\emph{What is the probability that the Gilbert graph has at least twice its expected number of edges?}\\[2ex]
	\emph{What is the probability that the Gilbert graph has at most half its expected number of edges?}
\end{center}
\bigskip

These seemingly innocuous questions have an intriguing connection to large deviations of heavy-tailed sums. Indeed, suppose that $\{W_i\}$ is an equal-volume partition of $W$ such that the diameter of each $W_i$ is at most 1. Then, letting $Z_i := |X \cap W_i|$ denote the number of Poisson points in $W_i$, the edges entirely inside $W_i$ contribute $Z_i (Z_i - 1)/2$ to the total edge count. Since $Z_i$ follows a Poisson distribution, the tail probability $\P(Z_i \ge n)$ is of the order $\exp(-c\, n \log n)$ for some constant $c > 0$. Hence, the tails of $Z_i^2$ are of the order $\exp(- c\, \sqrt n \log n/2)$ and {therefore $Z_i^2$ does not have exponential moments.} This is critical to note, because it means that the problem at hand is tightly related to large deviations of heavy-tailed sums, where large deviations typically come from extreme realizations of the largest summand \cite{asmussen,rw,rev}. \\

The analysis and simulations in this paper rely on two methods: {\em  conditional Monte Carlo (MC)}  and \emph{importance sampling}; see Chapters 9.4 and 9.7 of \cite{handbook}, respectively. Importance sampling techniques have been in use for the analysis of rare events in the setting of the Erd\H{o}s--R\'enyi graph \cite{bhamidi}, which is different from the Gilbert graph considered here. Other than that, there has been very little literature on the topic.  Note that the same analysis can be extended to  Gilbert graphs with the threshold not equal to $1$ by modifying the size of the window $W$ and the intensity $\la$ of the Poisson point process. \\

The rest of the presentation is organized in three parts. First, in Section \ref{smallKSec}, we explore the potential of conditional Monte Carlo in a one-dimensional setting, where we can frequently  derive explicit closed-form expressions. Surprisingly, when moving to the far end of the upper tails, it is possible to avoid simulation altogether, as we derive a fully analytic representation. Then, in Section \ref{rbSec}, we move to higher dimensions. Here, we apply conditional MC to remove the randomness coming from the random number of nodes. Finally, in Section \ref{isSec}, we present a further refinement of the conditional MC estimator, by combining it with importance sampling using a Strauss-type Gibbs process. To ease notation, we henceforth identify the Gilbert graph with its edge set, so that $|G(X \cap W)|$ yields the number of edges in the Gilbert graph on $X \cap W$.

\section{Conditional MC in dimension 1}
\label{smallKSec}

%
%
In this section, we consider the one-dimensional setting. More precisely, we consider a line segment $W = [0, w]$ as sampling window. In Section \ref{edgeSec}, we describe a specific conditional MC scheme that leads to estimators for the rare-event probability of the number of edges being small. In a simulation study in Section \ref{smallSimSec}, we show that these new estimators are substantially better than a crude MC approach.\\

Then, Section \ref{missEdgeSec} discusses rare events corresponding to the number of missing edges, i.e., the number of point-pairs that are not connected by an edge. The analysis is motivated from the observation that the Erd\H os--R\'enyi graph with edge probability $p \in [0, 1]$ exhibits a striking duality with its complement. Specifically, the missing edges of this graph again form an Erd\H os--R\'enyi graph but now with probability $1 - p$. In Section \ref{missEdgeSec}, we point out that in the Gilbert graph such a duality is much more involved.
We still elucidate how to compute the probability of observing no missing edges or precisely one missing edge. \\

%
%
In this section, we assume that the points $\{X_i\}_{i \ge 1}$ of the Poisson point process on $[0, \ff)$ are ordered according to their occurrence; that is, $X_i \le X_j$, whenever $i \le j$. 

%
%
\subsection{Few edges}
\label{edgeSec}
Henceforth, let 
$$E_{\leq k} := \{|G(X \cap [0, w])| \le k\}$$
denote the event that the number of edges in the Gilbert graph on $X\cap [0,w]$ is at most $k \ge 0$. For fixed $k$ and large $w$, the probability
$$p_{\leq k} := \P(E_{\leq k} )$$
becomes small, and we discuss how to leverage both the natural ordering on the real half-line and the independence property of the Poisson point process to derive a refined estimator. \\

We focus only on the cases $k = 0, 1$. In principle, the methods could be extended to cover estimation of probabilities of the form $p_{\leq k}$ for $k \ge 2$. However, for large values of $k$ the combinatorial analysis becomes quickly highly involved; see Remark~\ref{rem:Kgeq2} for more details.

%
%
\subsubsection{No edges}
\label{edge0Sec}

To begin with, let $k = 0$. That is, we analyze the probability that all vertices in $X\cap [0, w]$ are isolated in the sense that their vertex degree is 0. The key idea for approaching this probability is to note that $E_0 := E_{\le 0}$ occurs if and only if $X\cap [X_1, (X_1 +1)\wedge w] = \es$ and the Gilbert graph restricted to $X\cap [X_1 + 1, w]$ does not contain edges; see Figure \ref{isol_fig}. Here, we adhere to the convention that $[a, b] = \es$ if $a > b$.

%
%
\begin{center}
\begin{figure}[!h]
	\begin{tikzpicture}[scale=2]

	\draw (0, 0)--(5, 0);
	\draw[very thick] (0, -.05)--(0,.05);
	\draw[very thick] (5, -.05)--(5,.05);

	\draw (.84, -1)--(5, -1);
	\draw[very thick] (.84, -1.05)--(.84,-.95);
	\draw[very thick] (5, -1.05)--(5, -.95);
	\draw[very thick] (1.54, -1.05)--(1.54,-.95);
	\draw[blue, ultra thick] (.84, -1)--(1.54, -1);

	\draw[red, ultra thick] (1.54, -1)--(5, -1);

\fill (.84, 0) circle (.9pt);
\fill (1.80, 0) circle (.9pt);
\fill (2.84, 0) circle (.9pt);
\fill (4.17, 0) circle (.9pt);

\fill (1.80, -1) circle (.9pt);
\fill (2.84, -1) circle (.9pt);
\fill (4.17, -1) circle (.9pt);

\coordinate[label={0}] (A) at (0, -0.35);
\coordinate[label={$X_1$}] (A) at (.84, -0.35);
\coordinate[label={$w$}] (A) at (5, -0.35);

\coordinate[label={$X_1$}] (A) at (.84, -1.35);
\coordinate[label={$X_1 + 1$}] (A) at (1.54, -1.35);
\coordinate[label={$w$}] (A) at (5, -1.35);

\end{tikzpicture}
	\caption{ For $G(X\cap [0, w]) = \es$, the blue interval may not contain points of $X$ and the Gilbert graph restricted to the red interval may not contain edges.}
	\label{isol_fig}
\end{figure}
\end{center}

According to the Palm theory for one-dimensional Poisson point processes, the process 
$$X^{(1)} := (X - X_1) \cap [1, \ff)$$
again forms a homogeneous Poisson point process, which is independent of $X_1$; see \cite[Theorem 7.2]{poisBook}. In particular, writing 
$$\FF_1 := \s(X_1, X^{(1)})$$ 
for the $\s$-algebra generated by $X_1$ and $X^{(1)}$ allows for a partial computation of $p_0 := p_{\le0}$ via conditional MC \cite[Chapter 9.4]{handbook}. More precisely, when computing $\P(E_0\,|\,\FF_1)$  we explicitly throw away the information from the configuration of $X$ inside the interval $[X_1, X_1 + 1]$.

\begin{theorem}[No edges]
	\label{rao0Thm}
	Suppose that $w \ge 1$. Then,
	\begin{align}
		\label{iso0Eq}
		\P(E_0\,|\,\FF_1) = e^{-\la((w - X_1)_+ \wedge 1)} \one\{G(X^{(1)} \cap [1, w - X_1]) = \es\} \quad \text{almost surely}.
	\end{align}
\end{theorem}
\begin{proof}
	First, $X_1$ is isolated if there are no further vertices in the interval $[X_1, (X_1 + 1) \wedge w]$. Moreover, after conditioning on $X_1$, the remaining vertices form a Poisson point process in $[(X_1 + 1) \wedge w, w]$; see \cite[Theorem 7.2]{poisBook}. Therefore, 
	\begin{align*}
		\P(E_0\,|\,\FF_1) &= \P\big(X \cap (X_1, (X_1 + 1) \wedge w) = \es\,|\, \FF_1\big)\one\{G(X^{(1)}\cap [1, w - X_1]) = \es\}\\
		&= e^{-\la(((X_1 + 1) \wedge w) - X_1)_+} 	\one\{G(X^{(1)} \cap [1, w - X_1]) = \es\}
	\end{align*}
	as asserted.\qed
\end{proof}

In other words, invoking the Rao--Blackwell theorem \cite{b2}, Theorem \ref{rao0Thm} showcases the right-hand side of identity \eqref{iso0Eq} as an attractive candidate for estimating $p_0$ via conditional Monte Carlo. The Rao--Blackwell theorem is a powerful tool in situations where $p_0$ is not available in closed form. Note that elementary properties of the conditional expectation imply that the conditional MC estimator $\P(E_0\,|\,\FF_1)$ is unbiased and exhibits smaller variance than the crude MC estimator $\one\{E_0 \}$. \\

Moreover, the right-hand side of identity \eqref{iso0Eq} features another indicator of an isolation event. Hence, it becomes highly attractive to refine the estimator further by proceeding iteratively. To make this precise, we define an increasing sequence $X_1^* \le X_2^* \le \cdots$ of points of $X$ recursively as follows. First, $X_1^* = X_1$ denotes the left-most point of $X$. Next, once $X_m^*$ is available,
$$X_{m + 1}^* := \inf\{X_i \in X:\, X_i \ge X_m^* + 1\}$$
denotes the first point of $X$ to the right of $X_m^* + 1$. Then, the event $E_0$ occurs if none of the intervals $[X_i^*, X_i^* + 1]$ contains points from $X$; see Figure \ref{isol_it_fig}.
%
%
\begin{figure}[!h]
	\begin{tikzpicture}[scale=2]

	\draw (0, 0)--(5, 0);
	\draw[very thick] (0, -.05)--(0,.05);
	\draw[very thick] (5, -.05)--(5,.05);

	\draw (.0, -1)--(5, -1);
	\draw[very thick] (.00, -1.05)--(.00,-.95);
	\draw[very thick] (.84, -1.05)--(.84,-.95);
	\draw[very thick] (5, -1.05)--(5, -.95);

	\draw[blue, ultra thick] (.84, -1)--(1.54, -1);
	\draw[very thick] (1.54, -1.05)--(1.54,-.95);

	\draw[blue, ultra thick] (1.8, -1)--(2.5, -1);
	\draw[very thick] (2.5, -1.05)--(2.5,-.95);

	\draw[blue, ultra thick] (2.84, -1)--(3.54, -1);
	\draw[very thick] (3.54, -1.05)--(3.54,-.95);

	\draw[blue, ultra thick] (4.17, -1)--(4.87, -1);
	\draw[very thick] (4.87, -1.05)--(4.87,-.95);

\fill (.84, 0) circle (.9pt);
\fill (1.80, 0) circle (.9pt);
\fill (2.84, 0) circle (.9pt);
\fill (4.17, 0) circle (.9pt);

\fill (1.80, -1) circle (.9pt);
\fill (2.84, -1) circle (.9pt);
\fill (4.17, -1) circle (.9pt);

\coordinate[label={0}] (A) at (0, -0.35);
\coordinate[label={$X_1$}] (A) at (.84, -0.35);
\coordinate[label={$w$}] (A) at (5, -0.35);

\coordinate[label={0}] (A) at (0, -1.35);
\coordinate[label={$X_1^* = X_1$}] (A) at (1.04, -1.35);
\coordinate[label={$X_2^*$}] (A) at (1.84, -1.35);
\coordinate[label={$X_3^*$}] (A) at (2.84, -1.35);
\coordinate[label={$X_4^* = X_*$}] (A) at (4.37, -1.35);
\coordinate[label={$w$}] (A) at (5, -1.35);

\end{tikzpicture}
	\caption{
		For $G(X\cap [0, w]) = \es$, the blue intervals may not contain any points.}
	\label{isol_it_fig}
\end{figure}
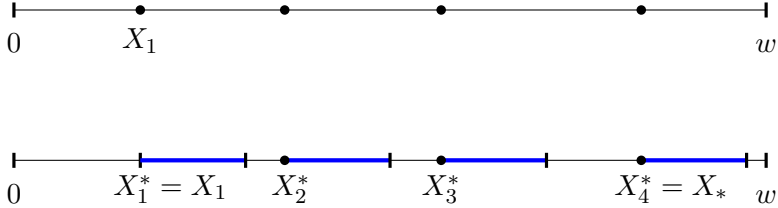

Of particular interest is the last index
$$I_* := \sup\{i \ge 1:\, X_i^* \le w\},$$
where $X_i^*$ remains inside $[0, w]$, together with the associated point 
$$X_* := X_{I_*}^*.$$
If $X_1 > w$, we set $I_* = 0$ and $X_* = w$. Let 
$$\FF^* := \s\big(X_1^*, X_2^*, \dots \big),$$ 
be the $\s$-algebra generated by $\{X_i^* : i \geq 1\}$.
%
%
\begin{theorem}[No edges -- iterated]
	\label{raoCor}
	Suppose that $w \ge 1$. Then, 
	$$
		\P(E_0 \,|\FF^*) = e^{-\la((I_* - 1)_+ + (w - X_*) \wedge 1)} \quad \text{almost surely},
		$$
\end{theorem}
\begin{proof}
	To prove the claim, we first define the shifted process
	$$X^{(m)} := (X - X_m^*)\cap [1, \ff)$$
	and write 
$$\FF_m := \s\big(X_1^*,\dots, X_m^*, X^{(m)}\big)$$ 
for the $\s$-algebra generated by $X_1^*, \dots, X_m^*$ and $X^{(m)}$. 
Observe that $\FF_1 \supseteq \FF_2 \supseteq \cdots \supseteq \FF^*$. 
In particular, by the tower property of conditional expectation,
	\begin{align*}
		\P\big(E_0 \,|\,\FF^*\big) &= \sum_{m \ge 0}\E\big[\one\{I_* = m\} \one\{E_0\} \,|\,\FF^*\big] = \sum_{m \ge 0}\one\{I_* = m\}\P(E_0 \,|\,\FF^*)\\
		                     &= \sum_{m \ge 0}\one\{I_* = m\}\E\big[\P(E_0 \,|\, \FF_m)\,|\,\FF^*\big]\one\{I_* \geq m\}.
	\end{align*}
	Hence, it suffices to show that for every $m \ge 0$, 
	\begin{equation}
	\label{eqn:condAndI}
	\P(E_0 \,|\, \FF_m)\one\{I_* \ge m\} = e^{-\la((m - 1)_+ + (w - X_m^*) \wedge 1)}\one\{G(X^{(m)}\cap [1, w - X_m^*]) = \es\}\one\{I_* \ge m\},
	\end{equation}
	because $X^{(I_*)} \cap [1, w - X^*] = \es$.
	To achieve this goal, we proceed by induction on $m$. For $m = 0$ and $m = 1$ we are in the setting of Theorem \ref{rao0Thm}. To pass from $m$ to $m + 1$, the induction hypothesis yields that
	\begin{align}
		\P(E_0 \,|\,\FF_{m + 1}) \one\{I_* \ge m + 1\}&= \E\big[\P(E_0 \,|\,\FF_{m})\,|\,\FF_{m + 1}\big]\one\{I_* \ge m + 1\}\nonumber\\
&= e^{-\la m} \P\big(G(X^{(m)}\cap [1, w - X_m^*]) = \es\,|\,\FF_{m + 1}\big)\one\{I_* \ge m + 1\}. \label{eqn:first}
	\end{align}
	Since $X_{m + 1}^*$ is the first point of $X$ after $X_m^* + 1$, by applying Theorem \ref{rao0Thm} to the Poisson point process $X^{(m)}$, we obtain
	\begin{align}
		\P(G(X^{(m)} \cap [1, w - X_m^*])\hspace{-.1cm} = \hspace{-.1cm}\es|\FF_{m + 1})\hspace{-.1cm} = \hspace{-.1cm}e^{-\la((w - X_{m + 1}^*) \wedge 1)} \one\{G(X^{(m + 1)}\cap [1, w - X_{m + 1}^*]) \hspace{-.1cm}= \hspace{-.1cm}\es\}. \label{eqn:second}
	\end{align}
	Combining \eqref{eqn:first} and \eqref{eqn:second} yields the assertion.\qed
\end{proof}
%
%
The conditional MC estimator from Theorem \ref{raoCor} leads to Algorithm \ref{rao0Alg}. Here, $\ms{Exp}(\la)$ is an exponential random variable with parameter $\la$ that is independent of everything else.
\begin{algorithm}[htb]
\SetAlgoSkip{}
\DontPrintSemicolon
\SetKwInput{Input}{input}\SetKwInOut{Output}{output}
\Input{Number $N \ge 1$ of MC runs and Poisson intensity $\la > 0$.}
\Output{Conditional MC estimator of $p_0$.}

	$\ms{sum} \lw 0$
	
\For{$i\le N$}
	{
	Draw $Z$ from $\ms{Exp}(\la)$\\
	$p \lw 1$

	\While{$Z \leq w$}
	{
		$p \lw p e^{-\la ((w - Z) \wedge 1)}$\\
	Draw $Y$ from $\ms{Exp}(\la)$\\
	$Z \lw Z + 1 + Y$
	}

$\ms{sum} \lw \ms{sum} + p$
}

	\KwRet{$\ms{sum}/N$}
\caption{Conditional MC estimator for $p_0$}
	\label{rao0Alg}
\end{algorithm}
%
%
\subsubsection{At most one edge} 
\label{edge1Sec}
Here, let $k = 1$; i.e., we propose an estimator for the probability $p_1 := p_{\leq1}$ that the Gilbert graph on $X\cap [0, w]$ has at most one edge.
Let 
$$I^+ := \inf\{i \ge 2:\, X_i - X_{i - 1} \le 1\}$$
be the index of the first point of $X$ whose predecessor is at distance at most~1. Putting $X^+ := (X - X_{I^+}) \cap [1, \ff)$, Figure \ref{one_fig} illustrates that the event $E_{\leq 1} $ is equal to the intersection of the events $\{X_{I^+ + 1} \ge X_{I^+} + 1\}$ and $\{G(X^+\cap [1, w - X_{I^+}]) = \es\}$.
%
%
\begin{figure}[!h]
	\begin{tikzpicture}[scale=2]

	\draw (0, -1)--(5, -1);
	\draw[very thick] (0, -1.05)--(0,-.95);
	\draw[very thick] (5, -1.05)--(5,-.95);

	\draw (3.17, -1)--(5, -1);
	\draw[very thick] (5, -1.05)--(5, -.95);
	\draw[very thick] (3.87, -1.05)--(3.87,-.95);
	\draw[blue, ultra thick] (3.17, -1)--(3.87, -1);

	\draw[red, ultra thick] (3.87, -1)--(5, -1);

\fill (.84, -1) circle (.9pt);
\fill (1.80, -1) circle (.9pt);
\fill (2.84, -1) circle (.9pt);
\fill (3.17, -1) circle (.9pt);
\fill (4.17, -1) circle (.9pt);

\fill (3.17, -1) circle (.9pt);
\fill (4.17, -1) circle (.9pt);

\coordinate[label={0}] (A) at (0, -1.35);
\coordinate[label={$X_1$}] (A) at (.84, -1.35);

	\coordinate[label={$X_{I^+} $}] (A) at (3.17, -1.35);
	\coordinate[label={$X_{I^+} + 1$}] (A) at (3.87, -1.35);
\coordinate[label={$w$}] (A) at (5, -1.35);

\end{tikzpicture}
	\caption{ For $|G(X\cap [0, w])| \le 1$, the blue interval may not contain any points and the Gilbert graph restricted to the red interval may not contain edges.}
	\label{one_fig}
\end{figure}

Moreover, we write 
$$\FF^+ := \s(X_1, X_2, \dots, X_{I^+}, X^+)$$
for the $\s$-algebra generated by $X_1, X_2, \dots, X_{I^+}$ and $X^+$. 

\begin{theorem}[At most one edge]
	\label{rao1Thm}
	Suppose that $w \ge 1$. Then, 
	\begin{align}
		\label{iso1Eq}
		\P(E_{\leq 1} \,|\,\FF^+) = e^{-\la((w - X_{I^+})_+ \wedge 1)} \one\{G(X^+\cap [1, w - X_{I^+}]) = \es\}, \quad \text{almost surely}.
	\end{align}
\end{theorem}
\begin{proof}
	Since the proof is very similar to that of Theorem \ref{rao0Thm}, we only point to the most important ideas. Equation \eqref{iso1Eq} obviously holds for $X_{I^+} > w$. For the case $X_{I^+} \le w$, relying again on the Palm theory of the one-dimensional Poisson process, we have
	\begin{align*}
		\P(E_{\leq 1} \,|\,\FF^+) &= \P\big(X \cap (X_{I^+ }, (X_{I^+ } + 1) \wedge w] = \es\,|\, \FF^+\big)\one\big\{G(X^+ \cap [1, w - X_{I^+}]) = \es\big\}\\
		&= e^{-\la\big(((X_{I^+} + 1) \wedge w) - X_{I^+}\big)_+} \one\big\{G(X^+ \cap [1, w - X_{I^+}]) = \es \big\},
	\end{align*}
	as asserted.\qed
\end{proof}

Similarly to Section \ref{edge1Sec}, Theorem \ref{rao1Thm} yields a conditional MC estimator, which we describe in Algorithm \ref{rao1Alg}. 

\SetKwRepeat{Repeat}{repeat}{until}
\begin{algorithm}[htb]
\SetAlgoSkip{}
\DontPrintSemicolon
\SetKwInput{Input}{input}\SetKwInOut{Output}{output}
\Input{Number $N \ge 1$ of MC runs and Poisson intensity $\la > 0$.}
\Output{Conditional MC estimator of $p_1$.}

	$\ms{sum} \lw 0$
	
\For{$i\le N$}
	{
	Draw $Z$ from $\ms{Exp}(\la)$\\
	$ p \lw 1$\\
	\Repeat{$Y \le 1$}
	{
	Draw $Y$ from $\ms{Exp}(\la)$\\
	$Z \lw Z + Y$
	}
	\While{$Z \leq w$}
	{
		$ p \lw  p e^{-\la ((w - Z) \wedge 1)}$\\
	Draw $Y$ from $\ms{Exp}(\la)$\\
	$Z \lw Z + 1 + Y$
	}

$\ms{sum} \lw \ms{sum} +  p$
}

	\KwRet{$\ms{sum}/N$}
\caption{Conditional MC estimator for $p_1$}
	\label{rao1Alg}
\end{algorithm}

\begin{remark}[Symmetric window]
\label{rem:sym}
	The methods described above could also be applied for a Poisson point process in a symmetric interval of the form $[-w/2, w/2]$. Then, in addition to $I_*$ and $I^+$, we would need to take into account the corresponding quantities located to the left of the origin.
\end{remark}

\begin{remark}[$k \ge 2$]
\label{rem:Kgeq2}
The method to handle $k=0,1$ could certainly be extended to larger $k \ge 2$, but the configurational analysis would quickly become very involved. For instance, for $k = 2$ we would first need to require that the interval $[X_{I_* - 1}, X_{I_* - 1} + 1]$ contains only the point $X_{I_*}$. Next, if $X_{I_* + 1} \le X_{I_*} + 1$, then there may be no more edges to the right of $X_{I_* + 1}$. On the other hand, the analysis will be different if $X_{I_* + 1} > X_{I_*} + 1$, because then there can still be one edge in the graph to the right of $X_{I_* + 1}$.
\end{remark}

	%
	%
	\subsection{Simulations}
	\label{smallSimSec}
	In this section, we illustrate how to estimate the rare-event probabilities $p_0$ and $p_1$ via MC. After presenting the crude MC estimator, we illustrate how the conditional MC estimators described in Section \ref{edgeSec} improve the efficiency drastically.
In the simulation study, we estimate $p_0$ and $p_1$ for sampling windows of size $w \in \{5, 7.5, 10\}$ and Poisson intensity $\la = 2$. Both the crude MC and the conditional MC estimator are computed based on $N = 10^6$ samples.
	%
	%
	To estimate the rare-event probabilities $p_{\leq k}$ using crude MC, we draw iid samples $X^{(1)}, \dots, X^{(N)}$ of the Poisson point process on $[0, w]$ and set
$$ \pcmc_{\leq k} := \frac 1N \sum_{i \le N} \one\{|G(X^{((i)}\cap [0, w])| \le k\big\}$$
for the proportion of samples leading to a Gilbert graph with at most $k$ edges.\\

The estimates reported in Table \ref{cmcSmallTab} reveal that as the size of the sampling window grows, the rare-event probabilities decrease rapidly and that the estimators exhibit a high relative error.

%
%
     \begin{table}[!htpb]
	     \centering
		 \caption{Crude MC estimates for $p_0 = p_{\leq 0}$ and $p_{\leq 1}$ for sampling intervals of size $w \in \{5, 7.5, 10\}$ and Poisson intensity $\la = 2$. Standard deviations reported as $\pm$.}
     \label{cmcSmallTab}

 \begin{tabular}{lccr}
    $w$ & $\pcmc_{0}\, $ & $\pcmc_{\leq 1}\, $  \\
\hline
     $5$ & $4.056 \times 10^{-3} \pm 6.36 \times 10^{-5}$ & $1.676 \times 10^{-2} \pm 1.28 \times 10^{-4}$\\
     $7.5$ & $2.410 \times 10^{-4} \pm 1.55 \times 10^{-5}$ & $1.354 \times 10^{-4} \pm 3.68 \times 10^{-5}$ \\
     $10$ & $1.100 \times 10^{-5} \pm 3.32 \times 10^{-6}$ & $8.500 \times 10^{-5} \pm 9.23 \times 10^{-6}$\\ \ \\
\end{tabular}
     \end{table}
	%
	%
	Next, we estimate $p_{\leq k}$ for $k=0, 1$ with the conditional MC methods described in Algorithms~\ref{rao0Alg} and \ref{rao1Alg}. If $P^{(1)}, \dots, P^{(N)}$ denote the simulation outputs, then we set
$$ \pc_{\leq k} := \frac 1N \sum_{i \le N} P^{(i)}.$$

The corresponding estimates shown in Table \ref{raoSmallTab} highlight that the theoretical improvements over crude MC predicted from Theorems \ref{raoCor} and \ref{rao1Thm} also manifest themselves in the simulation study. This is particularly striking in the setting $k = 0$, where the variance can be reduced by several orders of magnitude.

%
%
     \begin{table}[!htpb]
	     \centering
	 \caption{Conditional MC estimates for $p_0 = p_{\leq 0}$ and $p_{\leq 1}$ for sampling intervals of size $w \in \{5, 7.5, 10\}$ and Poisson intensity $\la = 2$. Variance improvements in comparison to the crude estimator in parentheses.}
     \label{raoSmallTab}

 \begin{tabular}{lccr}
    $w$ & $\pc_0\,$ & $\pc_{\leq 1} \,$  \\
\hline
	 $5$ & $4.148 \times 10^{-3} \pm 1.41 \times 10^{-5}$ (20.30) & $1.670 \times 10^{-2} \pm 7.63 \times 10^{-4}$ (2.8)\\
	 $7.5$ & $2.244 \times 10^{-4} \pm 1.06 \times 10^{-6}$ (216.3) & $1.304 \times 10^{-4} \pm 1.87 \times 10^{-5}$ (3.9) \\
	 $10$ & $1.296 \times 10^{-5} \pm 1.06 \times 10^{-7}$ (984.9) & $9.637 \times 10^{-5} \pm 4.56 \times 10^{-6}$ (4.1)\\ \ \\
\end{tabular}
	     \end{table}

%
%
\subsection{Few missing edges}
\label{missEdgeSec}

We now focus on computing the rare-event probabilities of having few missing edges. More precisely, we write
$$M_w := \binom{|X\cap [0, w]|}2 - |G(X\cap [0, w])|$$
for the number of edges of $G(X\cap [0, w])$ that are missing from the complete graph on the vertices $X\cap [0, w]$. We write 
$$p'_{\leq k} := \P(M_w \le k)$$ 
for the probability that at most $k \ge 0$ edges are missing. Surprisingly, this seemingly more complicated task is more accessible than the probability of seeing few edges considered in Section \ref{edgeSec}, as both $p'_0 = p'_{\leq 0}$ and $p'_{\leq 1}$ are amenable to closed-form expressions.\\

For $p'_0$, the key insight is to note that $M_w = 0$ if and only if $X \cap [X_1 + 1, w] = \es$; see Figure~\ref{miss_fig}.
%
%
\begin{figure}[!h]
	\begin{tikzpicture}[scale=2]
	\draw (0, 0)--(5, 0);
	\draw[very thick] (0, -.05)--(0,.05);
	\draw[very thick] (5, -.05)--(5,.05);

	\draw[very thick] (5, -0.05)--(5, .05);
	\draw[very thick] (3.54, -0.05)--(3.54,.05);

	\draw[red, ultra thick] (3.54, -0)--(5, -0);

\fill (2.84, 0) circle (.9pt);
\fill (3.17, 0) circle (.9pt);
\fill (3.22, 0) circle (.9pt);

\coordinate[label={0}] (A) at (0, -0.35);
\coordinate[label={$X_1$}] (A) at (2.84, -0.35);
\coordinate[label={$w$}] (A) at (5, -0.35);

\coordinate[label={$X_1 + 1$}] (A) at (3.54, -0.35);

\end{tikzpicture}
	\caption{ For $M_w = 0$, the red interval may not contain any points.}
	\label{miss_fig}
\end{figure}
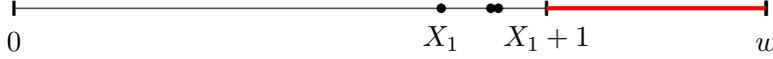

\begin{theorem}[No missing edges]
	\label{cliq0Thm}
	Suppose that $w \ge 1$. Then, 
	$$p'_0 = e^{- \la (w - 1)} + (w -1)\la e^{-\la(w - 1)}.$$
	
\end{theorem}
\begin{proof}
	As observed in the above remark, we need to compute $\P(X\cap [X_1 + 1, w] = \es)$. Hence, invoking the void probability for a Poisson point process,
	\begin{align*}
		p'_0 = \int_0^\ff\la e^{-\la x_1} \P(X\cap [x_1 + 1, w] = \es)\, \d x_1
		= e^{- \la (w - 1)} + \int_0^{w - 1}\la e^{-\la x_1} e^{-\la(w - 1 - x_1)} \, \d x_1,
	\end{align*}
	which equals $
		 e^{- \la (w - 1)} + (w -1)\la e^{-\la(w - 1)}$, 
	as asserted.\qed
\end{proof}

Next, we compute the probability of observing at most one missing edge.

\begin{theorem}[At most one missing edge]
	\label{cliq1Thm}
	Suppose that $w \ge 2$. Then,
	$$p'_{\leq 1} = p'_0 + \frac{\la^2(w - 2)^2}2 e^{-\la w} + (w - 3/2) \la^2e^{-\la (nw- 1)}.$$
\end{theorem}
\begin{proof}
	We decompose $p'_{\leq 1}$ as 
	$$p'_{\leq 1} = p'_0 + \P\big(M_w = 1, X\cap [X_1 + 2, w] \ne \es\big) + \P\big(M_w = 1, X\cap [X_1 + 2, w] = \es\big)$$
	and compute the second and third probability separately. This corresponds to Case 1 and Case 2 below. Note that under the event $\{X\cap [X_1 + 2, w] \ne \es\}$, we have $|X\cap [0, w]| = 2$, since more points would imply at least two missing edges. \\ \ \\
	{\bf Case 1. $ X\cap (X_1, X_1 + 2] = \es$ and $|X\cap [ X_1 + 2, w]|= 1$}\\
	\medskip
	Conditioning on $X_1$, the probability of this event becomes
	$$\int_0^{w - 2} \la e^{-\la x_1}\P\big(X \cap (x_1, x_1 + 2] = \es, |X \cap[ x_1 + 2, w]|= 1\big) \, \d x_1.$$
	Inserting the void probability for the Poisson process, we arrive at

	\begin{align*}
		 \int_0^{w - 2} \la e^{-\la x_1}e^{-2\la} \la (w - 2 - x_1) e^{-\la (w - 2 - x_1)}  \d x_1
		= \la^2 e^{-\la w}\int_0^{w - 2} (w - 2 - x_1) \, \d x_1
		= \frac{\la^2(w - 2)^2}2 e^{-\la w}.
	\end{align*}

	It remains to treat the case $X\cap [X_1 + 2, w] = \es$. First, note that conditioned on $X_1 = x_1$, the point process $X \cap [x_1, w]$ is Poisson. Thinking of this Poisson point process to be formally extended to $-\ff$ to the left, we write $X_1'$ for the first Poisson point to the left of $w$. Then, Figure~\ref{miss2_fig} illustrates that the event $\{M_w = 1\} \cap \{ X\cap [X_1 + 2, w] = \es\}$ is equivalent to
	$X_1' \in [X_1 + 1, X_1 + 2 ]$ and $X\cap (X_1, X_1' - 1] = \es$.
%
%
\begin{figure}[!h]
	\begin{tikzpicture}[scale=2]

	\draw (0, -1)--(5, -1);
	\draw[very thick] (0, -1.05)--(0,-.95);
	\draw[very thick] (5, -1.05)--(5,-.95);

	\draw[red, ultra thick] (2.84, -1.00)--(3.13, -1.00);
	\draw[blue, ultra thick] (3.54, -1.00)--(4.24, -1.00);

\fill (2.84, -1) circle (.9pt);
\fill (3.83, -1) circle (.9pt);

\coordinate[label={0}] (A) at (0, -1.35);

\coordinate[label={$X_1$}] (A) at (2.84, -1.35);
\coordinate[label={$X_1'$}] (A) at (3.83, -1.35);
\coordinate[label={$w$}] (A) at (5, -1.35);

\end{tikzpicture}
	\caption{Configuration where $M_w = 1$ and $ X\cap [X_1 + 2, w] = \es$. Here, $[X_1, X_1' - 1]$ is in red and $[X_1 + 1, X_1 + 2]$ in blue.}
	\label{miss2_fig}
\end{figure}
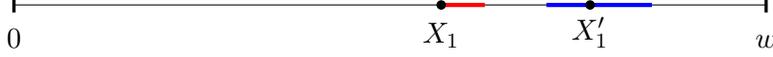

For Case 2, we distinguish between the cases $X_1 \le w - 2$ and $X_1 \in [w - 2, w - 1]$.\\ \ \\
	{\bf Case 2a. $X_1 \le w - 2$, $X_1' \in [X_1 + 1, X_1 + 2]$ and $X \cap (X_1, X_1' - 1] = \es$}\\[1ex]
	Then, we compute the desired probability as 
	\begin{align*}
		\int_0^{w - 2}\la e^{-\la x_1}\int_{x_1 + 1}^{x_1 + 2}\la e^{-\la (w - x_1')}\P\big(X\cap (x_1, x_1' - 1] = \es\big) \, \d x_1' \, \d x_1
		= \int_0^{w - 2}\int_{x_1 + 1}^{x_1 + 2}\la^2 e^{-\la (w - 1)} \, \d x_1' \, \d x_1,
	\end{align*}
	which equals $(w - 2) \la^2e^{-\la (w - 1)}.$\\
	{\bf Case 2b. $X_1 \in [w - 2, w -1]$, $X_1' \in [X_1 + 1, w]$ and $X\cap (X_1, X_1' - 1] = \es$}\\[1ex]
	Finally, 
	\begin{align*}
		\int_{w - 2}^{w - 1}\la e^{-\la x_1}\int_{x_1 + 1}^w\la e^{-\la (w - x_1')}\P\big(X\cap(x_1, x_1' - 1] = \es\big) \, \d x_1' \, \d x_1
		= \int_{w - 2}^{w - 1}\int_{x_1 + 1}^w\la^2 e^{-\la (w - 1)} \, \d x_1' \, \d x_1,
	\end{align*}
which equals $= \frac{\la^2}2 e^{-\la (w - 1)}.$
	Assembling the different cases together concludes the proof.\qed
	\end{proof}

\section{Conditional MC in higher dimensions}
\label{rbSec}
%
%
In Section \ref{smallKSec}, we analyzed rare events related to few edges or few missing edges in a one-dimensional setting. There, the natural ordering of the Poisson points was pivotal to derive closed-form expressions for conditional probabilities. Now, we proceed to higher dimensions and also consider more general deviations from the mean number of edges. First, we again illustrate that substantial variance reductions are possible through a surprisingly simple conditional MC method. Loosely speaking, the Poisson point process consists of 1) an infinite sequence of random points in the window determining the locations of points and 2) a Poisson random variable determining the number of points in the sampling window. We use that after conditioning on the spatial locations, the rare-event probability is available in closed form. This type of Poisson conditioning is novel in a spatial rare-event estimation context, but has strong ties to approaches appearing earlier in seemingly unrelated problems. More precisely, for instance in reliability theory, related conditional MC schemes lead to spectacular variance reductions~\cite{lom,slava}.\\

%
%
The rest of this section is organized as follows. First, Section \ref{rbLowSec} describes how to estimate the rare event probabilities related to too few and too many edges relative to their expected number through conditional MC. Then, Section \ref{rbNumSec} presents a simulation study illustrating that this estimator can reduce the estimation variance by several orders of magnitude.

%
%
\subsection{Conditioning on a spatial component}
\label{rbLowSec}
We consider a full-dimensional sampling window $W \subset \R^d$ and rare events of the form 
$$F_{< a} := \{|G(X \cap W)| < (1 - a)\mu\}\quad \text{ and } \quad F_{>a} := \{|G(X \cap W)| > (1 + a) \mu\},$$
where $\mu := \E[|G(X \cap W)|]$ denotes the mean (i.e., expected) number of edges, which can be estimated through simulations. Alternatively, if $W$ is so large that edge effects can be neglected, then the Slivnyak--Mecke formula \cite[Theorem 4.4]{poisBook} gives the approximation $\mu \approx \frac12|W|\la^2 \k_d$, where $\k_d$ denotes the volume of the unit ball in $\R^d$.\\

The key idea for developing a conditional MC scheme is to use the explicit construction of a Poisson point process in a bounded sampling window. More precisely, let $X_\ff = \{X_n\}_{n \ge 1}$ be an iid family of uniform random points in $W$ and $K$ be an independent Poisson random variable with parameter $\la |W|$. Then, $\{X_n\}_{n \le K}$ is a Poisson point process in $W$ with intensity $\la$, \cite[Theorem 3.6]{poisBook}.\\

In conditional MC, we use the fact that the rare-event probabilities can be computed in closed form after we condition on the locations $X_\ff$. More precisely, we let 
$$K_{<a} := \inf\{k \ge 1:\, |G(\{X_1,\dots, X_k\} \cap W)| \ge (1 - a)\mu\}$$
denote the first time where the Gilbert graph on the nodes $\{X_1, \dots, X_k\}$ contains at least $(1 - a)\mu$ edges and refer to Figure \ref{quant_fig} for an illustration.

%
%
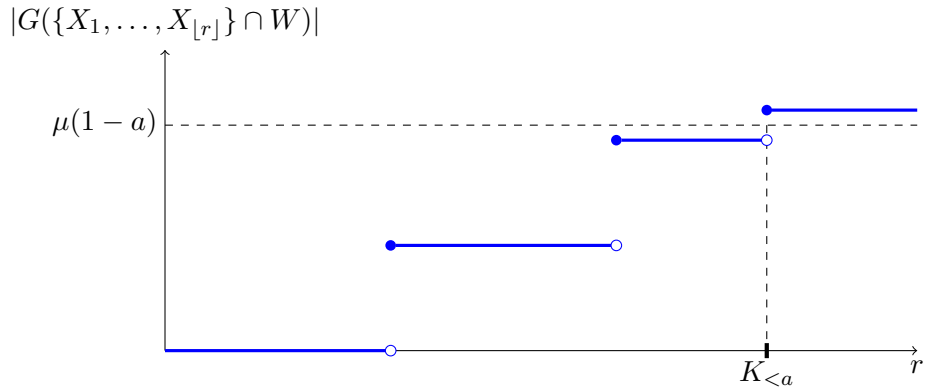
\begin{figure}[!h]
	        \begin{tikzpicture}[scale=2]

	\draw[->] (0, 0)--(5, 0);
	\draw[->] (0, 0)--(0, 2);
	\draw[dashed] (0, 1.5)--(5, 1.5);
	\draw[dashed] (4, 1.5)--(4, 0);
	\draw[ultra thick] (4, -.05)--(4, .05);


	\draw[blue, very thick] (4, 1.6)--(5, 1.6);
	\draw[blue, very thick] (3, 1.4)--(4, 1.4);
	\draw[blue, very thick] (1.5, 0.7)--(3, 0.7);
	\draw[blue, very thick] (0, 0.0)--(1.5, 0);

	\fill[white] (4, 1.6) circle (1pt);
	\fill[white] (4, 1.4) circle (1pt);
	\fill[white] (3, 1.4) circle (1pt);
	\fill[white] (3, 0.7) circle (1pt);
	\fill[white] (1.5, 0.7) circle (1pt);
	\fill[white] (1.5, 0.0) circle (1pt);

	\fill[blue] (4, 1.6) circle (1pt);
	\draw[blue] (4, 1.4) circle (1pt);
	\fill[blue] (3, 1.4) circle (1pt);
	\draw[blue] (3, 0.7) circle (1pt);
	\fill[blue] (1.5, 0.7) circle (1pt);
	\draw[blue] (1.5, 0.0) circle (1pt);

\coordinate[label=below:{$r$}] (A) at (5, 0);
	\coordinate[label={$|G(\{X_1, \dots, X_{\lfloor r\rfloor}\} \cap W)|$}] (A) at (0, 2);
	\coordinate[label={$\mu (1 - a)$}] (A) at (-.4, 1.35);
	\coordinate[label=below:{$K_{<a}$}] (A) at (4, 0);
\end{tikzpicture}
			\caption{Conceptual illustration of $K_{<a}$.}
			        \label{quant_fig}
\end{figure}

Similarly, let 
$$K_{>a} := \sup\{k \ge 1:\, |G(\{X_1,\dots, X_k\} \cap W)| \le (1 + a)\mu\}$$
denote the largest $k$ such that the Gilbert graph on the nodes $\{X_1, \dots, X_k\}$ contains at most $(1 + a)\mu$ edges. Then,
\begin{equation}
	\label{condDimEq}
	\begin{aligned}
		\P(F_{<a}) &= \E[\P(F_{<a}\,|\,X_\ff)] = \E\big[\Pois(K_{<a}(X_\ff))\big],\\
		\P(F_{>a}) &= \E[\P(F_{>a}\,|\,X_\ff)] = \E\big[1 - \Pois(K_{>a}(X_\ff - 1))\big],
	\end{aligned}
\end{equation}
where ${\Pois:\Z_{\ge 0} \to [0, 1]}$ denotes the cumulative distribution function of a Poisson random variable with parameter $\la |W|$. 

%
%
\subsection{Numerical results}
\label{rbNumSec}
We sample planar homogeneous Poisson point processes $\Xs = \{\Xs_i\}_{i \ge 1}$, $\Xm = \{\Xm_i\}_{i \ge 1}$ and $\Xl= \{\Xl_i\}_{i \ge 1}$ with intensity $\la = 2$ in windows of size $20 \times 20$, $25 \times 25$, and $30 \times 30$, respectively. Here, $\ms S, \ms M$, and  $\ms L$ stand for small, medium, and large, respectively.\\

In Section \ref{intrSec}, we mentioned that a major challenge in devising efficient estimators for rare-event probabilities related to the edge count comes from a qualitatively different tail behavior: light on the left, heavy on the right. In other words, we expect that in the left tail, we see changes throughout the sampling window, whereas in the right tail, a singular particular configuration in a small part of the window is sufficient to induce the rare event. We recall that the left tail of a random variable $Z$ refers to the probabilities $\P(Z \le r)$ for small $r$ and the right tail refers to the probabilities $\P(Z \ge r)$ for large $r$. \\

Although on a bounded sampling window, the theoretical difference between the left and the right tail is subtle, we illustrate in Table \ref{tailTab} that it does become visible when considering the quantiles $Q_\a$ and $Q_{1 - \a}$ for the number of edges if $\a$ is small. Here, the empirical quantiles for $10^6$ samples of the edge counts in a $(20 \times 20)$-window are shown. For instance the $1\%$-quantile is $16.4\%$ lower than the mean, which is a similar deviation as the $18\%$ exceedance of the $99\%$-quantile. However, when moving to the $0.01\%$-quantile, then it is $25.2\%$ lower than the mean, whereas the corresponding $99.99\%$ quantile exceeds the mean by $30.0\%$. These figures are an early indication of the difference in the tails that will reappear far more pronouncedly in the numerical results concerning the estimation of the rare-event probabilities $\P(F_{<0.2})$ and $\P(F_{>0.2})$. 
%
%
\begin{table}[!htpb]
	\centering
	\caption{Empirical quantiles for the number of edges in a $(20 \times 20)$-window: absolute values (upper two rows) and relative deviation from the mean $\mu$ (lower two rows).}
	\label{tailTab}

	\begin{tabular}{lllr}
		$\a$ &  $10^{-2}$ & $10^{-3}$ & $10^{-4}$   \\
		\hline
		$Q_\a$ & 2012 & 1892 & 1800\\
		$Q_{1 - \a}$ &  2841 & 2999 & 3130\\[1ex]
		$\frac{Q_\a - \mu}\mu$ & $-0.164$ & $-0.214$ & $-0.252$\\
		$\frac{Q_{1 - \a} - \mu}\mu$ & 0.180 & 0.246 & 0.300\\[1ex]
\end{tabular}
\end{table}

\vspace{-.5cm}
In the rest of the section, we estimate the rare-event probabilities
$$q_{< 0.2} := \P(F_{<0.2}) \quad \text{ and }\quad q_{> 0.2} := \P(F_{>0.2})$$
corresponding to $20 \%$ deviations from the mean.
Here, we draw $N = 10^5$ samples of $X_\ff$. Then, taking into account the representation \eqref{condDimEq}, we set 
\begin{align*}
	\pc_{< 0.2} &:= \frac 1N\sum_{i \le N}\Pois\big(K_{<a}(X_\ff(i))\big),\\
	\pc_{> 0.2} &:= \frac 1N\sum_{i \le N}\big(1 - \Pois\big(K_{>a}(X_\ff(i) - 1)\big)\big).
\end{align*}
The variances of the crude MC estimators are equal to $q_{< 0.2} (1 - q_{< 0.2})$. \\

We report the estimates $\pc_{< 0.2}$ and $\pc_{> 0.2}$ in Table \ref{cmcTab}. First, we see that the exceedance  probabilities $\pc_{> 0.2}$ are always smaller than the corresponding undershoot probabilities $\pc_{< 0.2}$. This supports the preliminary impression of the difference in the tail behavior hinted at in Table \ref{tailTab}. Moreover, in all examples conditional MC reduces the estimation variance massively and the efficiency gains become more pronounced the rarer the event.

%
%
\begin{table}[!h]
	\centering
	\caption{Estimates for $q_{< 0.2}$ and $q_{> 0.2}$ based on conditional MC for different window sizes based on $N = 10^5$ samples. Variance improvements in comparison to the crude estimator in parentheses.}
	\label{cmcTab}

	\begin{tabular}{lllr}
		&  $\pc_{< 0.2}$ & $\pc_{> 0.2}$   \\
		\hline
		$\Xs$ & $2.023 \times 10^{-3} \pm 6.98 \times 10^{-6}\; (414.8)$ & $5.118 \times 10^{-3} \pm 1.63 \times 10^{-5}\; (193.7)$\\
		$\Xm$ & $1.542 \times 10^{-4} \pm 7.05 \times 10^{-7}\; (3106.4)$ & $6.764 \times 10^{-4} \pm 2.77 \times 10^{-6}\; (878.6)$ \\
		$\Xl$ & $6.912 \times 10^{-6} \pm 4.19 \times 10^{-8}\; (39,415.8$) & $6.242 \times 10^{-5} \pm 3.24 \times 10^{-7}\; (5,911.1)$\\
		\hline
	\end{tabular}
\end{table}
\bigskip
\bigskip

\section{Importance sampling}
\label{isSec}

The conditional MC estimators constructed in Section \ref{rbSec} take into account that an atypically large number of Poisson points leads to a Gilbert graph exhibiting considerably more edges than expected. However, not only the number but also the location of points play a  pivotal role. For instance if the points tend to repel each other, then we would typically observe fewer edges. Similarly, clustered point patterns should induce  more edges.\\

In order to implement these insights, we resort to the technique of \emph{importance sampling} \cite[Chapter 9.7]{handbook}. That is, we draw samples from a point process with distribution $\Q$ for which the rare event becomes typical and then correct the estimation bias by weighting with likelihood ratios. In Sections \ref{isLowSec} and \ref{isUpSec} below, we explain how to implement these steps through configuration-dependent birth-death processes reminiscent of the Strauss process from spatial statistics \cite{rasmus}. Finally, in Section \ref{isNumSec}, we illustrate in a simulation study how importance sampling of the spatial locations reduces the estimation variance further.

%
%
\subsection{Lower tails}
\label{isLowSec}
%
%
The key observation is that under the rare event of seeing exceptionally few edges, we expect a repulsion between points. More precisely, the most likely reason for the rare event are changes to the configuration of the underlying Poisson point process throughout the entire sampling window. The large-deviation analysis of \cite{yukLDP2} suggests to perform importance sampling where, instead of considering the distribution $\P$ of the a priori Poisson point process, we draw samples according to a different stationary point process with distribution $\Q$ such that under $\Q$, the original rare event becomes typical and whose deviation from $\P$, as measured through the Kullback--Leibler divergence $h(\Q\,|\,\P)$, is minimized. We implement this repulsion by a dependent thinning inspired from the Strauss process.\\

Here, we start from a realization of the Gilbert graph on $n_0 = \lfloor\la|W|\rfloor$ iid points $\{X_1, \dots, X_{n_0}\}$, and then, we thin out points successively. An independent thinning of points would give rise to uniformly distributed locations without interactions.  In the importance sampling, we thin instead via a configuration-dependent birth-death process; see e.g., Chapter 9.7 of \cite{handbook}.\\

%

To describe the death mechanism more precisely, we draw inspiration from the Strauss process and choose the probability $p_i$ to remove point $X_i$ proportional to $\g^{\deg(X_i)}$, where $\deg(X_i)$ denotes the degree of $X_i$ in the Gilbert graph and $\g > 1$ is a parameter of the algorithm. Algorithm \ref{is_alg} shows the pseudo-code leading to the importance sampling estimator $\pp_{<a}$ for $q_{<a}$. To understand the principle behind Algorithm \ref{is_alg}, we briefly expound on the general approach in importance sampling, and refer the reader to Chapter 9.7 of \cite{handbook} for an in-depth discussion. \\

As mentioned above, when thinning out points independently until the number of edges in the Gilbert graph falls below $\mu (1 - a)$, we would arrive at the random variable $K_{<a}$ from Section \ref{rbSec}. However, thinning out according to a configuration-dependent probability distorts its distribution towards a probability measure $\Q$ that is in general different from the true distribution $\P$. Still, by construction, $\Q$ is absolutely continuous with respect to $\P$, and we let $\rho := \d \P / \d \Q$ be the likelihood ratio. Then, from a conceptual point of view, Algorithm \ref{is_alg} first draws $N\ge1$ iid samples $K_{<a}^{(1)}, \dots, K_{<a}^{(N)}$ from the distorted distribution $\Q$ with associated likelihood ratios $\rho_1, \dots, \rho_N$, and then computes
$$\pp_{<a} := \frac 1N \sum_{i \le N} \rho_i \Pois(K_{<a}(i)).$$

\begin{algorithm}[htb]
\SetAlgoSkip{}
\DontPrintSemicolon
\SetKwInput{Input}{input}\SetKwInOut{Output}{output}
\Input{Number $N \ge 1$ of MC runs; parameter $\g$; Poisson intensity $\la$; sampling window $W$.}
\Output{Importance sampling estimator $\pp_{<a}$ for $q_{<a}$.}

$n_0 \lw \lfloor\la|W|\rfloor$

	$\ms{sum} \lw 0$

\For{$i\le N$}
	{
$X \lw \{X_1, \dots, X_{n_0}\}$ iid uniform in $W$

$\rho \lw 1$

	\While{$|G(X \cap W)| \le \mu (1 - a)$}
	{
draw $X_i \in X$ with probability $r_i = \g^{\deg(X_i)} / \sum_{X_j \in X}\g^{\deg(X_j)}$

		$\rho \lw \rho / (|X| r_i)$

	$X \lw X \setminus \{X_i\}$

		}

		$\ms{sum} \lw \ms{sum} + \rho \Pois(|X|)$
		}

	\KwRet{$\ms{sum}/N$}
\caption{Importance sampling estimator $\pp_{<a}$ for $q_{<a}$}
	\label{is_alg}
\end{algorithm}

Intuitively, we would like to choose the thinning parameter $\g > 1$ such that the number of edges in the rare event should match the expected number of edges under the thinning. To develop a heuristic for this choice, we consider a Strauss process, where we restrict to the two-dimensional setting to allow for an accessible presentation.
 Thus,
\begin{align}
	        \label{strEq}
			(1 - a) \mu = (\la^\Str)^2\int_0^1 \pi r \rho^\Str(r) \d r,
\end{align}
where $\la^\Str > 0$ and $\rho^\Str:\,[0,\ff) \to [0,\ff)$ denote the intensity and pair-correlation function of the Strauss process, respectively \cite{rasmus}. In contrast to the Poisson setting, neither $\la^\Str$ nor $\rho^\Str$ are available in closed form. Still, both quantities admit accurate saddle-point approximations $\lap > 0$ and $\rp:\, [0, \ff) \to [0,\ff)$ that are ready to implement in the Strauss case \cite{nair}.\\

%
%
More precisely, $\lap$ is the unique positive solution of the equation
$$\lap G e^{\lap G} = \la G,$$
where $G = (1 - \b) \pi$ and $\b = \log(\g)$, see \cite{nair}. Then, recalling that we work in a planar setting, we put
$$\rp(r) := \b \exp\big((1 - \b)^2 b(r/2) \lap\big),$$
where
$$b(r/2) := 2  \cos^{-1}(r/2) - r \sqrt{1 - (r/2)^2} $$
denotes the intersection area of two unit disks at distance $r$.  Inserting these approximations into \eqref{strEq} and solving the resulting implicit equation yields $\b \approx \log(1.018)$ and a value of $\g \approx 1.018$.\\

%
%
\subsection{Upper tails}
\label{isUpSec}
Similar to the lower tails, we can strengthen the estimator by combining conditional MC with importance sampling on the spatial locations. For the upper tails, it is natural to devise an importance sampling scheme favoring clustering rather than repulsion. From the point of view of large deviations, this phenomenon was considered in \cite{harel}. There, it is shown that all excess edges come from a ball of size 1 containing a highly dense configuration of Poisson points.  More precisely, the method is particularly powerful in settings where the rare event is the epitome of a condensation phenomenon. That is, a peculiar behavior of the point process in a small part in the window becomes the most likely explanation of the rare event. As laid out in \cite{harel}, at least in the asymptotic regime of large deviations, the rare event of observing too many edges is governed by the above-described condensation phenomenon.  \\

We propose to take the clustering into account via a birth mechanism favoring the generation of points in areas that would lead to a large number of additional edges. To ease implementation, the density of the birth mechanism is discretized and remains constant in bins of a suitably chosen grid. The density in a bin at position $x \in W$ is proportional to $\g^{n(x)}$, where $\g > 1$ is a parameter governing the strength of the clustering and $n(x)$ denotes the number of Poisson points in a suitable neighborhood around $x$, such as the bin containing $x$ together with all adjacent bins.\\

Similar to the lower tails case, a subtle point in this approach pertains choosing the parameter $\g > 1$. Unfortunately, an attractive Strauss process is ill-defined in the entire Euclidean space, so that the saddle-point approximation from Section \ref{isLowSec} does not apply. In Section \ref{isNumSec} below, we therefore rely on a pilot run suggesting $\g = 1.01$ as a good choice for further variance reduction. Although in this pilot run, the number of simulations is small, and estimates of the variance are still volatile, we found that it provides a good indication for simulations on a larger scale.

%
%
\subsection{Numerical results}
\label{isNumSec}

Section \ref{rbNumSec} revealed that even with a sample size of $N = 10^5$ the conditional MC estimators still exhibit a considerable relative error. Now, we illustrate that importance sampling may be an appealing option to reduce this error.\\

%
%
Table \ref{poisTab} reports the estimates $\pp_{<0.2}$ and $\pp_{>0.2}$ for different sizes of the sampling window as in Section~\ref{rbNumSec}. In the left tail, we see massive variance improvements when compared to the crude MC estimator. In the right tail, the gains are substantial, but a little less pronounced. This matches the intuition that the root of the rare events in the right tails should be a condensation phenomenon, which goes against the heuristic of changing the point process homogeneously throughout the window. 
         \begin{table}[!htpb]
		 \centering
			   \caption{Estimates for $q_{<0.2}$ and $q_{>0.2}$ based on importance sampling with Strauss-type processes for different window sizes  based on $N = 10^5$ samples. Variance improvements in comparison to the crude estimator in parentheses.}
			            \label{poisTab}

		 \begin{tabular}{lllr}
			 &  $\pp_{<0.2}$ & $\pp_{>0.2}$   \\
			 \hline
			 $\Xs$ & $2.025 \times 10^{-3} \pm 6.22 \times 10^{-6} (523.3)$ & $5.125 \times 10^{-3} \pm 1.57 \times 10^{-5} (207.9)$\\
			 $\Xm$ & $1.544 \times 10^{-4} \pm 6.16 \times 10^{-7} (4071.0)$ & $6.744 \times 10^{-4} \pm 2.66 \times 10^{-6} (951.8)$ \\
			 $\Xl$ & $6.935 \times 10^{-6} \pm 3.63 \times 10^{-8} (52,665.8$) & $6.240 \times 10^{-5} \pm 3.09 \times 10^{-7} (6,537.22)$\\[1ex]
		 \end{tabular}
				             \end{table}

\section*{Acknowledgments}
Major parts of this research were carried out while Christian Hirsch was on a research visit at the University of Queensland and while Sarat B.~Moka was on a research visit at Ulm University. We thank both hosts for their hospitality and the Australian Research Council Centre of Excellence for Mathematical and Statistical Frontiers (ACEMS) for the support under grant number CE140100049.

\bibliography{lit}

\begin{thebibliography}{10}
\providecommand{\url}[1]{{#1}}
\providecommand{\urlprefix}{URL }
\expandafter\ifx\csname urlstyle\endcsname\relax
  \providecommand{\doi}[1]{DOI~\discretionary{}{}{}#1}\else
  \providecommand{\doi}{DOI~\discretionary{}{}{}\begingroup
  \urlstyle{rm}\Url}\fi

\bibitem{asmussen}
Asmussen, S., Kroese, D.P.: Improved algorithms for rare event simulation with
  heavy tails.
\newblock Adv. in Appl. Probab. \textbf{38}(2), 545--558 (2006)

\bibitem{nair}
Baddeley, A., Nair, G.: Approximating the moments of a spatial point process.
\newblock Stat \textbf{1}(1), 18--30 (2012)

\bibitem{bhamidi}
Bhamidi, S., Hannig, J., Lee, C.Y., Nolen, J.: The importance sampling
  technique for understanding rare events in {E}rd{\H{o}}s-{R}\'{e}nyi random
  graphs.
\newblock Electron. J. Probab. \textbf{20}, no. 107, 30 (2015)

\bibitem{b2}
Billingsley, P.: Probability and Measure, third edn.
\newblock J.~Wiley \& Sons, Inc., New York (1995)

\bibitem{rw}
Blanchet, J., Glynn, P.: Efficient rare-event simulation for the maximum of
  heavy-tailed random walks.
\newblock Ann. Appl. Probab. \textbf{18}(4), 1351--1378 (2008)

\bibitem{harel}
Chatterjee, S., Harel, M.: Localization in random geometric graphs with too
  many edges.
\newblock Ann. Probab. \textbf{48}(2), 574--621 (2020)

\bibitem{handbook}
Kroese, D.P., Taimre, T., Botev, Z.I.: Handbook of Monte Carlo Methods.
\newblock J.~Wiley \& Sons, New York (2013)

\bibitem{poisBook}
Last, G., Penrose, M.D.: Lectures on the Poisson Process.
\newblock Cambridge University Press, Cambridge (2017)

\bibitem{lom}
Lomonosov, M., Shpungin, Y.: Combinatorics of reliability {M}onte {C}arlo.
\newblock Random Structures Algorithms \textbf{14}(4), 329--343 (1999)

\bibitem{rasmus}
M{\o}ller, J., Waagepetersen, R.P.: Statistical Inference and Simulation for
  Spatial Point Processes.
\newblock CRC, Boca Raton (2004)

\bibitem{rev}
Rojas-Nandayapa, L.: A review of conditional rare event simulation for tail
  probabilities of heavy tailed random variables.
\newblock Bol. Soc. Mat. Mexicana (3) \textbf{19}(2), 159--182 (2013)

\bibitem{yukLDP2}
Sepp{\"a}l{\"a}inen, T., Yukich, J.E.: Large deviation principles for
  {E}uclidean functionals and other nearly additive processes.
\newblock Probab. Theory Related Fields \textbf{120}(3), 309--345 (2001)

\bibitem{slava}
Vaisman, R., Kroese, D.P., Gertsbakh, I.B.: Improved sampling plans for
  combinatorial invariants of coherent systems.
\newblock IEEE Trans. Rel. \textbf{65}(1), 410--424 (2015)

\end{thebibliography}
\bibliographystyle{spmpsci}

\end{document}